\documentclass[12pt]{article}
\usepackage{amsfonts, amssymb, amsthm}

\newtheorem{theorem}{Theorem}[section]
\newtheorem{question}{Question}[section]
\newtheorem{corollary}{Corollary}[section]
\theoremstyle{definition}

\newtheorem{remark}{Remark}[section]

\newcommand{\I}{\mathcal{I}}
\newcommand{\B}{\mathcal{B}}
\newcommand{\PP}{\mathcal{P}}
\newcommand{\UU}{\mathcal{U}}

\newcommand\NN{{\mathbb N}}

\newcommand\w{{\omega}}
\newcommand{\F}{\mathcal{F}}

\begin{document}

\title{Ultrafilters on $G$-spaces}
\author{O.V. Petrenko, I.V. Protasov}
\date{}

\maketitle

\begin{abstract}
For a discrete group $G$ and a discrete $G$-space $X$, we identify the Stone-\v{C}ech compactifications $\beta G$ and $\beta X$ with the sets of all ultrafilters on $G$ and $X$, and apply the natural action of $\beta G$ on $\beta X$ to characterize large, thick, thin, sparse and scattered subsets of $X$. We use $G$-invariant partitions and colorings to define $G$-selective and $G$-Ramsey ultrafilters on $X$. We show that, in contrast to the set-theoretical case, these two classes of ultrafilters are distinct. We consider also universally thin ultrafilters on $\omega$, the $T$-points, and study interrelations between these ultrafilters and some classical ultrafilters on $\omega$.

\


\


\end{abstract}

By a $G$-space, we mean a set $X$ endowed with the action $G\times X \to X:(g,x)\mapsto gx$ of a group $G$. All $G$-spaces are supposed to be transitive: for any $x,y \in X$, there exists $g\in G$ such that $gx=y$. If $X=G$ and the action is the group multiplication, we say that $X$ is a regular $G$-space.

Several intersting and deep results in combinatorics, topological dynamics and topological algebra, functional analysis were obtained by means of ultrafilters on groups (see \cite{b5,b6,b7,b12,b27,b28}).

The goal of this paper is to systematize some recent and prove some new results concerning ultrafilters on $G$-spaces, and point out the key open problems.

In sections 1,2 and 3, we keep together all necessary definitions of filters, ultrafilters and the Stone-\v{C}ech compactification $\beta X$ of the discrete space $X$. We extend the action of $G$ on $X$ to the action of $\beta G$ on $\beta X$, characterize the minimal invariant subsets of $\beta X$, define the corona $\check{X}$ of $X$ and the ultracompanions of subsets of $X$.

In section 4, we give ultrafilter charecterizations of large, thick, thin, sparse and scattered subsets of $X$.

In section 5, we use $G$-invariant partitions and colorings to define $G$-selective and $G$-Ramsey ultrafilters on $X$, and show that, in contrast to the set-theoretical case, these two classes are essentially different.

In section 6, we use countable group of permutatious of $\w=\{0,1,\ldots\}$ to define thin ultrafilters on $\w$. We prove that some classical ultrafilters on $\w$ (for example, $P$- and $Q$-points) are thin ultrafilters.

We conclude the paper, showing in section 7, how all above result can be considered and interpreted in the frames of general asymptology.

\section{Filters and ultrafilters}
A family $\F$ of subsets of a set $X$ is called {\it filter} if $X\in\F, \varnothing\notin\F$ and
$$A,B\in \F, A\subseteq C \Rightarrow A\cap B \in \F, C\in \F$$
The family of all fillters on $X$ is partially ordered by inclusion $\subseteq$. A filter $\UU$ that is maximal in this ordering is called an {\it ultrafilter}. Equivalentely, $\UU$ is ultrafilter if $A\cup B\in \UU$ implies $A\in \UU$ or $B\in\UU$. This characteristic of ultrafilters plays the key role in the Ramsey Theory: to prove that, under any finite partition of $X$, at least one cell of the partition has a given property, it suffices to construct an ultrafilter $\UU$ such that each member of $\UU$ has this property.

An ultrafilter $\UU$ is called {\it principal} if $\{x\}\in\UU$ for some $x\in X$. Non-principal ultrafilters are called {\it free} and the set of all free ultrafilters on $X$ is denoted by $X^*$.

We endow a set $X$ with the discrete topology. The Stone-\v{C}ech compactification $\beta X$ of $X$ is a compact Hausdorff space such that $X$ is a subspace of $\beta X$ and any  mapping $f:X\to Y$ to a compact Hausdorff space $Y$ can be extended to the continuous mapping $f^\beta: \beta X \to Y$. To work with $\beta X$, we take the points of $\beta X$ to be the ultrafilters on $X$, with the points of $X$ identified with the principal ultrafilters, so $X^*=\beta X \setminus X$.

The topology of $\beta X$ can be defined by stating that the sets of the form $\overline{A}=\{p \in \beta X: A \in p\} $, where $A$ is a subset of $X$, are base for the open sets. For a filter $\varphi$ on $X$, the set $\overline{\varphi} = \{\overline{A}: A\in \varphi\}$ is closed in $\beta X$, and each non-empty closed subset of $\beta X$ is of the form $\overline{\varphi}$ for an appropriate filter $\varphi$ on $X$.

\section{The action of $\beta G$ on $\beta X$}

Given a $G$-space $X$, we endow $G$ and $X$ with the discrete topologies and use the universal property of the Stone-\v{C}ech compactification to define the action of $\beta G$ on $\beta X$.

Given $g\in G$, the mapping $x \mapsto gx : X \to \beta X$ extends to the continuous mapping
$$p \mapsto gp : \beta X \to \beta X.$$

We note that $gp = \{gP : P\in p\}$, where $gP=\{gx : x \in P\}$.

Then, for each $p \in \beta X$, we extend the mapping $g\mapsto gp: G \to \beta X$ to the continuous mapping
$$q\mapsto qp: \beta G \to \beta X.$$

Let $q\in \beta G$ and $p \in \beta X$. To describe a base for the ultrafilter $qp\in \beta X$, we take any element $Q\in q$ and, for every $g\in Q$, choose some element $P_g \in p$. Then $\bigcup_{g\in Q}gP_g \in qp$ and the family of subsets of this form is a base for $qp$.

By  the construction, for every $g\in G$, the mapping $p\mapsto gp:\beta X \to \beta X$ is continuous and, for every $p\in \beta X$, the mapping $q\mapsto qp: \beta G \to \beta X$ is continuous. In the case of the regular $G$-space $X$, $X=G$, we get well known (see \cite{b7}) extention of multiplication from $G$ to $\beta G$ making $\beta G$ a compact right topological semigroup. For plenty applications of the semigroup $\beta G$ to combinatorics and topological algebra see \cite{b6,b7,b12,b28}. It should be marked that, for any $q, r\in \beta G$, and $p\in \beta X$, we have $(qr)p = q(rp)$ so semigroup $\beta G$ acts on $\beta X$.

Now we define the main technical tool for study of subsets of $X$ by means of ultrafilters.

Given a subset $A$ of $X$ and an ultrafilter $p\in \beta X$ we define the {\em $p$-companion} of $A$ by
$$A_p=\{\overline{A}\cap Gp\}= \{gp: g\in G, A\in gp\}.$$
Systematically, $p$-companions will be used in section 4. Here we demonstrate only one appication of $p$-companion to characterize minimal invariant subsets of $\beta X$.
A closed subset $S$ of $\beta X$ is called {\it invariant} if $g\in G$ and $p\in S$ imply $gp\in S$. Clearly, $S$ is invariant if and only if $(\beta G)p \subseteq S$ for each $p\in S$. Every invariant subset $S$ of $\beta X$ contains minimal by inclusion invariant subset. A subset $M$ is minimal invariant if and only if $M=(\beta G)p$ for each $p\in S$. In the case of the regular $G$-space, the minimal invariant subsets coincide with minimal left ideals of $\beta G$ so the following theorem generalizes Theorem 4.39 from \cite{b7}.
\begin{theorem} Let $X$ be a $G$-space and let $p\in \beta X$. Then $(\beta G)p$ is minimal invariant if and only if, for every $A\in p$, there exists a finite subset $F$ of $G$ such that $G=FA_p$. \end{theorem}
\begin{proof}
We suppose that  $(\beta G)p$ is a minimal invariant subset and take an arbitary $r\in \beta G$. Since  $(\beta G)rp = (\beta G)p$ and $p\in (\beta G)p$, there exists $q_r\in \beta G$ such that $q_rrp=p$. Since $A\in q_rrp$, there exists $x_r\in G$ such that $A\in x_rrp$ so $x_r^{-1}A\in rp$. Then we choose $B_r\in r$ such that $\overline{x_r^{-1}A}\supseteq \overline{B_r}p$ and consider the open cover $\{\overline{B_r}: r \in \beta G\}$ of $\beta G$. By compactness of $\beta G$, there is its finite subcover $\{\overline{B_{r_1}}, \ldots, \overline{B_{r_n}}\}$, so $G=B_{r_1}\cup\ldots\cup B_{r_n}$. We put $F^{-1} = \{x_{r_1}, \ldots, x_{r_n}\}$. Then $G=(FA)_p$ and it suffices to observe that $(FA)_p=FA_p$.

To prove the converse statement, we suppose on the contrary that $(\beta G)p$ is not minimal and choose $r\in \beta G$ such that $p\notin (\beta G)rp$. Since $(\beta G)rp$ is closed in $\beta X$, there exists $A\in p$ such that $\overline{A}\cap (\beta G)rp = \varnothing$. It follows that $A\notin qrp$ for every $q\in \beta G$. Hence, $G\setminus A \in qrp$ for each $q\in \beta G$ and, in particular, $x(G\setminus A) \in rp$ for each $x\in G$. By the assumption, $gA_p\in
r$ for some $g\in G$ so $A\in g^{-1}rp$, $gA\in rp$ and we get a contradiction.
\end{proof}

\section{Dynamical equivalences and coronas}
For an infinite discrete $G$-space , we define two basic equivalences on the space $X^*$ of all free ultrafilter on $X$.

Given any $r,q \in X^*$, we say that $r, q$ are {\it parallel}  (and write $r\parallel q$) if there exists $g\in G$ such that $q=gr$. We denote by $\sim$ the minimal (by inclusion) closed in $X^*\times X^*$ equivalences on $X^*$ such that $\parallel\subseteq\sim$. The quotient $X^*/\sim$ is a compact Hausdorff space. It is called the corona of $X$ and is denoted by $\check{X}$.

For every $p\in X^*$, we denote by $\check{p}$ the class of the equivalence $\sim$ containing $p$, and say that $p, q\in X^*$ are corona equivalent if $\check{p} = \check{q}$. To detect whether two ultrafilters $p, q\in X^*$ are corona equivalent, we use $G$-slowly oscillating functions on $X$.

A function $h:X\to[0,1]$ is called {\it $G$-slowly oscillating} if, for any $\varepsilon>0$ and finite subset $K\subset G$, there exists a finite subset $F$ of $X$ such that
$$diam\; h(Kx)<\varepsilon,$$
for each $x\in X\setminus F$, where $diam\; h(Kx)=\sup\{|h(y)-h(z)|:y,z\in Kx\}$.

\begin{theorem}
Let $q, r\in X^*$. Then $\check{q} = \check{r}$ if and only if $h^\beta(r)=h^\beta(q)$ for every $G$-slowly oscillating function $h:X\to[0,1]$.
\end{theorem}

For more detailed information on dynamical equivalences and topologies of coronas see \cite{b14} and \cite{b1,b13,b17,b19}.

In the next section, for a subset $A$ of $X$ and $p\in X^*$, we use the {\it corona $p$-companion} of $A$
$$A_{\check{p}} = A^* \cap \check{p}.$$

\section{Diversity of subsets of $G$-spaces}

For a set $S$, we use the standard notation $[S]^{<\w}$ for the family of all finite subsets of $S$.

Let X be a $G$-space, $x\in X, A\subseteq X, K\in [G]^{<\w}$. We set
$$B(x, K) = Kx\cup\{x\}, B(A,K)=\bigcup_{a\in A}B(a,K),$$
and say that $B(x,K)$ is a {\it ball of radius} $K$ around $x$. For motivation of this notation, see the section 7.

Our first portion of definitions concerns the upward directed properties: $A\in\PP$ and $A\subseteq B$ imply $B\in \PP$.

A subset $A$ of a $G$-space $X$ is called
\begin{itemize}
  \item {\it large} if there exists $K\in [G]^{<\w}$ such that $X=KA$;
  \item {\it thick} if, for every $K\in [G]^{<\w}$, there exists $a\in A$ such, that $Ka\subseteq A$;
  \item {\it prethick} if there exists $F\in [G]^{<\w}$ such that $FA$ is thick.
\end{itemize}

In the dynamical terminology \cite{b7}, large and prethick subsets are known as syndedic and piecewise syndedic subsets.

\begin{theorem} For a subset $A$ of an infinite $G$-space $X$, the following statements hold:
\begin{itemize}
  \item[(i)] $A$ is large if and only if $A_p \ne \varnothing$ for each $p\in X^*$;
  \item[(ii)] $A$ is thick if and only if, there exists $p\in X^*$ such that $A_p=Gp$.
\end{itemize}
\end{theorem}
\begin{proof}
$(i)$ We suppose that $A$ is large and choose $F\in [G]^{<\w}$ such that $X=FA$. Given any $p\in X^*$, we choose $g\in F$ such that $gA\in p$. Then $A\in g^{-1}p$ and $A_p \ne \varnothing$.

To prove the converse statement, for every $p\in X^*$, we choose $g_p\in G$ such that $A\in g_pp$ so $g_p^{-1}A\in p$. We consider an open covering of $X^*$ by the subsets $\{g_p^{-1}A^* : p \in X^*\}$ and choose its finite subcovering $g_{p_1}^{-1}A^*, \ldots, g_{p_n}^{-1}A^*$. Then the set $H=X\setminus (g_{p_1}^{-1}A^*\cup \ldots\cup g_{p_n}^{-1}A^*)$ is finite. We choose $F\in [G]^{<\w}$ such that $H\subset FA$ and $\{g_{p_1}^{-1}, \ldots, g_{p_n}^{-1}\} \subset F$. Then $X=FA$ so
$A$ is large.

$(ii)$ We note that $A$ is thick if and only if $X\setminus A$ is not large and apply $(i)$.
\end{proof}

\begin{theorem} 
A subset $A$ of an infinite $G$-space $X$ is prethick if and only if there exists $p\in X^*$ such that $A\in p$ and $(\beta G)p$ is a minimal invariant subsets of $\beta X$.
\end{theorem}
\begin{proof}
The theorem was proved for regular $G$-spaces in \cite[Theorem 4.40]{b7}. This proof can be easily adopted to the general case if we use Theorem 2.1 in place of Theorem 4.39 from \cite{b7}.
\end{proof}
\begin{corollary}
For every finite partition of a $G$-space $X$, at least one cell of the partition is prethick.
\end{corollary}
\begin{remark}
For a subset $A$ of an infinite $G$-space $X$, we set 
$$\Delta(A)=\{g\in G: g^{-1}A\cap A \mbox{ is infinite}\}.$$
Let $\PP$ be a finite partition of $X$. We take $p\in X^*$ such that the set $(\beta G)p$ is minimal invariant and choose $A\in\PP$ such that $A\in p$. By Theorem 2.1, $A_p$ is large in $G$. If $g\in A_p$ then $g^{-1}A\in p$ and $A\in p$. Hence, $g^{-1}A\cap A$ is infinite, so $A_p\subseteq \Delta(A)$ and $\Delta(A)$ is large.
\end{remark}

In fact, this statement can be essentially strengthened: there is a function $f:\NN\to\NN$ such that, for every $n$-partition $\PP$ of a $G$-space $X$, there are $A\in \PP$ and $F\subset G$ such that $G=F\Delta(A)$ and $|F| \leqslant f(n)$. This is an old open problem (see the surveys \cite{b2,b22} whether the above statement is true with $f(n)=n$).

In the second part of the section, we consider the downward directed properties $A\in \PP$, $B\subseteq A$ imply $B\in P$) and present some results from \cite{b3,b23}
A subset $A$ of a $G$-space $X$ is called
\begin{itemize}
  \item {\it thin} if, for every $F\in [G]^{<\w}$, there exists $K\in [X]^{<\w}$, such that $B_A(a,F) = \{a\}$ for each $a\in A\setminus K$, where $B_A(a, F) = B(a,F)\cap A$;
  \item {\it sparse} if, for every infinite subset $Y$ of $X$, there exists $H\in [G]^{<\w}$ such that, for every $F\in [G]^{<\w}$, there is $y\in Y$ such that $B_A(y, F)\setminus B_A(y, H) = \varnothing$;
  \item {\it scattered} if, for every infinite subset $Y$ of $X$, there exists $H\in [G]^{<\w}$, such that, for every $F\in [G]^{<\w}$, there is $y\in Y$ such that $B_Y(a, F)\setminus B_Y(a, H) = \varnothing$.

\end{itemize}
\begin{theorem} 
For a subset $A$ of a $G$-space $X$, the following statements hold:
\begin{itemize}
  \item[(i)] $A$ is thin if and only if $|A_p| \leqslant 1$ for each $p\in X^*$;
  \item[(ii)] $A$ is sparse if and only if $A_p$ is finite for every $p\in X^*$;
\end{itemize}
\end{theorem}
Let $(g_n)_{n\in \w}$ be a sequence in $G$ and let $(x_n)_{n\in\w}$ be a sequence in $X$ such that
\begin{itemize}
  \item[(1)] $\{g_0^{\varepsilon_0} \ldots g_n^{\varepsilon_n}x_n : \varepsilon_i\in\{0,1\}\} \cap \{g_0^{\varepsilon_0} \ldots g_m^{\varepsilon_m}x_m : \varepsilon_i\in\{0,1\}\} = \varnothing $ for all distinct $m,n\in\w$;
  \item[(2)] $|\{g_0^{\varepsilon_0} \ldots g_n^{\varepsilon_n}x_n : \varepsilon_i\in\{0,1\}\}| = 2^{n+1}$ for every $n\in\w$.
\end{itemize}
We say that a subset $Y$ of $X$ is a {\it piecewise shifted $FP$-set} if there exist $(g_n)_{n\in \w}$, $(x_n)_{n\in\w}$ satisfying (1) and (2) such that
$$Y= \{g_0^{\varepsilon_0} \ldots g_n^{\varepsilon_n}x_n : \varepsilon_n\in\{0,1\}, n\in\w\}.$$
For definition of an $FP$-set in a group see \cite{b7}.
\begin{theorem} For a subset $A$ of a $G$-space $X$, the following statements are equivalent:
\begin{itemize}
  \item[(i)] $A$ is scattered;
  \item[(ii)] for every infinite subset $Y$ of $A$, there exists $p\in Y^*$ such that $Y_p$ is finite;
  \item[(iii)] $A_pp$ is discrete in $X^*$ for every $p\in X^*$;
  \item[(iv)] $A$ contains no piecewise shifted $FP$-sets.
\end{itemize}
\end{theorem}

\begin{theorem} Let $G$ be a countable group and let $X$ be a $G$-space. For a subset $A$ of $X$, the following statements hold:
\begin{itemize}
  \item[(i)] $A$ is large if and only if $A_{\check{p}} \ne \varnothing$ for each $p\in X^*$;
  \item[(ii)] $A$ is thick if and only if $\check{p} \subseteq A^*$ for some $p\in X^*$;
  \item[(iii)] $A$ is thin if and only if $|A_{\check{p}}| \leqslant 1$ for each $p\in X^*$;
  \item[(iv)] if $A_{\check{p}}$ is finite for each $p\in X^*$ then $A$ is sparse;
  \item[(v)] if, for every infinite subset $Y$ of $A$, there is $p\in Y^*$ such that $Y_{\check{p}}$ is finite then $A$ is scattered.
\end{itemize}
\end{theorem}

\begin{question}
Does the conversion of Theorem 4.5 $(iv)$ hold?
\end{question}

\begin{question}
Does the conversion of Theorem 4.5 $(v)$ hold?
\end{question}

\begin{remark}
If $G$ is an uncountable Abelian group then the corona $\check{G}$ is a singleton \cite{b13}. Thus, Theorem 4.5 does not hold (with $X=G$) for uncountable Abelian groups.
\end{remark}

\section{Selective and Ramsey ultrafilters on $G$-spaces}

We recall (see \cite{b4}) that a free ultrafilter $\UU$ on an infinite set $X$ is said to be {\it selective} if, for any partition $\PP$ of $X$, either one cell of $\PP$ is a member of $\UU$, or some member of $\UU$ meets each cell of $\PP$ in at most one point. Selective ultrafilters on $\omega$ are also known under the name {\it Ramsey ultrafilters} because $\UU$ is selective if and only if, for each colorings $\chi:[\omega]^2 \to \{0,1\}$ of $2$-element subsets of $\omega$, there exists $U\in \UU$ such that the restriction $\chi|_{[U]^2}\equiv const$.

Let $G$ be a group, $X$ be a $G$-space with the action $G\times X\to X$, $(g,x) \mapsto gx$. All $G$-spaces under consideration are supposed to be transitive: for any $x,y \in X$, there exists $g\in G$ such that $gx = y$. If $G=X$ and $gx$ is the product of $g$ and $x$ in $G$, $X$ is called a {\it regular $G$-space}. A partition $\PP$ of a $G$-space $X$ is {\it $G$-invariant} if $gP\in \PP$ for all $g\in G$, $P\in \PP$.

Let $X$ be an infinite $G$-space. We say that a free ultrafilter $\UU$ on $X$ is {\it $G$-selective} if, for any $G$-invariant partition $\PP$ of $X$, either some cell of $\PP$ is a member of $\UU$, or there exists $U\in \UU$ such that $|P\cap U|\leqslant 1$ for each $P\in\PP$. 

Clearly, each selective ultrafilter on $X$ is $G$-selective. Selective ultrafilters on $\omega$ exist under some additional to ZFC set-theoretical assumptions (say, CH), but there are models of ZFC with no selective ultrafilters on $\omega$.

Let $X$ be a $G$-space, $x_0\in X$. We put $St(x_0) = \{g\in G: gx_0 = x_0\}$ and identify $X$ with the left coset space $G/St(x_0)$. If $\PP$ is a $G$-invariant partition of $X = G/S, S = St(x_0)$, we take $P_0\in \PP$ such that $x_0\in P_0$, put $H = \{g\in G: gS \in P_0\}$ and note that the subgroup $H$ completely determines $\PP$: $xS$ and $yS$ are in the same cell of $\PP$ if and only if $y^{-1}x\in H$. Thus, $\PP = \{x(H/S): x\in L\}$ where $L$ is a set of representatives of the left cosets of $G$ by $H$.

\begin{theorem}
For every infinite $G$-space $X$, there exists a $G$-selective ultrafilter $\UU$ on $X$ in ZFC.
\end{theorem}

\begin{proof}
We take $x_0\in X$, put $S = St(x_0)$ and identify $X$ with $G/S$. We choose a maximal filter $\mathcal{F}$ on $G/S$ having a base consisting of subsets of the form $A/S$ where $A$ is a subgroup of $G$ such that $S\subset A$ and $|A:S|=\infty$. Then we take an arbitrary ultrafilter $\UU$ on $G/S$ such that $\mathcal{F}\subseteq \UU$.

To show that $\UU$ is $G$-selective, we take an arbitrary subgroup $H$ of $G$ such that $S\subseteq H$ and consider a partition $\PP_H$ of $G/S$ determined by $H$.

If $|H\cap A:S|=\infty$ for each subgroup $A$ of $G$ such that $A/S\in \mathcal{F}$ then, by the maximality of $\mathcal{F}$, we have $H/S\in\mathcal{F}$. Hence, $H/S\in\UU$.

Otherwise, there exists a subgroup $A$ of $G$ such that $A/S\in\mathcal{F}$ and $|H\cap A:S|$ is finite, $|H\cap A:S|=n$. We take an arbitrary $g\in G$ and denote $T_g = gH\cap A$. If $a\in T_g$ then $a^{-1}T_g\subseteq A$ and $a^{-1}T_g \subseteq H$. Hence, $a^{-1}T_g/S \subseteq A\cap H/S$ so $|T_g/S|\leqslant n$. If $x$ and $y$ determine the same coset by $H$, then they determine the same set $T_g$. Then we choose $U\in \UU$ such that $|U\cap x(H\cap A/S)|\leqslant 1$ for each $x\in G$. Thus, $|U\cap P|\leqslant 1$ for each cell $P$ of the partition $\PP_H$.
\end{proof}

The next theorem characterizes all $G$-spaces $X$ such that each free ultrafilter on $X$ is $G$-selective.

\begin{theorem}
Let $G$ be a group, $S$ be a subgroup of $G$ such that $|G:S|=\infty$, $X = G/S$. Each free ultrafilter on $X$ is $G$-selective if and only if, for each subgroup $T$ of $G$ such that $S\subset T \subset G$, either $|T:S|$ is finite or $|G:T|$ is finite.
\end{theorem}

Applying Theorem 2, we conclude that each free ultrafilter on an infinite Abelian group $G$ (as a regular $G$-space) is selective if and only if $G=
\mathbb{Z}\oplus F$ or $G = \mathbb{Z}_{p^\infty}\times F$, where $F$ is finite, $\mathbb{Z}_{p^\infty}$ is the Pr\"uffer $p$-group. In particular, each free ultrafilter on $\mathbb{Z}$ is $\mathbb{Z}$-selective.

For a $G$-space $X$ and $n\geqslant 2$, a coloring $\chi: [X]^n \to \{0,1\}$ is said to be {\it $G$-invariant} if, for any $\{x_1, \ldots, x_n\}\in [X]^n$ and $g\in G$, $\chi(\{x_1, \ldots, x_n\}) = \chi(\{gx_1, \ldots, gx_n\})$. We say that a free ultrafilter $\UU$ on $X$ is {\it $(G,n)$-Ramsey} if, for every $G$-invariant coloring $\chi:[X]^n\to \{0,1\}$, there exists $U\in \UU$ such that $\chi|_{[U]^n}\equiv const$. In the case $n=2$, we write "$G$-Ramsey" instead of $(G,2)$-Ramsey.

\begin{theorem}
For any $G$-space $X$, each $G$-Ramsey ultrafilter on $X$ is $G$-selective.
\end{theorem}

The following three theorems show that the conversion of Theorem 5.3 is very far from truth. Let $G$ be a discrete group, $\beta G$ is the Stone-\v{C}ech compactification of $G$ as a left topological semigroup, $K(\beta G)$ is the minimal ideal of $\beta G$.

\begin{theorem}
Each ultrafilter from the closure $cl\; K(\beta \mathbb{Z})$ is not $\mathbb{Z}$-Ramsey.
\end{theorem}

A free ultrafilter $\UU$ on an Abelian group $G$ is said to be a {\it Schur ultrafilter} if, for any $U\in \UU$, there are distinct $x,y\in U$ such that $x+y \in U$. 

\begin{theorem}
Each Schur ultrafilter $\UU$ on $\mathbb{Z}$ is not $\mathbb{Z}$-Ramsey.
\end{theorem}

A free ultrafilter $\UU$ on $\mathbb{Z}$ is called {\it prime} if $\UU$ cannot be represented as a sum of two free ultrafilters.

\begin{theorem}
Every $\mathbb{Z}$-Ramsey ultrafilter on $\mathbb{Z}$ is prime.
\end{theorem}

\begin{question}
Is each $\mathbb{Z}$-Ramsey ultrafilter on $\mathbb{Z}$ selective?
\end{question}

Some partial positive answers to this question are in the following two theorems.

\begin{theorem}
Assume that a free ultrafilter $\UU$ on $\mathbb{Z}$ has a member $A$ such that $|g+A \cap A|\leqslant 1$ for each $g\in \mathbb{Z} \setminus \{0\}$. If $\UU$ is $\mathbb{Z}$-Ramsey then $\UU$ is selective.
\end{theorem}

\begin{theorem}
Every $(\mathbb{Z}, 4)$-Ramsey ultrafilter on $\mathbb{Z}$ is selective.
\end{theorem}

All above results are from \cite{b9}.

\begin{remark}
Let $G$ be an Abelian group. A coloring $\chi: [G]^2 \to \{0,1\}$ is called a PS-{\it coloring} if $\chi(\{a,b\}) = \chi(\{a-g, b+g\})$ for all $\{a,b\} \in [G]^2$, equivalently, $a+b=c+d$ implies $\chi(\{a,b\}) = \chi(\{c,d\})$. A free ultrafilter $\UU$ on $G$ is called a $PS$-{\it ultrafilter} if, for any PS-coloring $\chi$ of $[G]^2$, there is $U\in \UU$ such that $\chi|_{[U]^2}\equiv const$. The following statements were proved in \cite{b18}, see also \cite[Chapter 10]{b6}.

If $G$ has no elements of order 2 then each PS-ultrafilter on $G$ is selective. A strongly summable ultrafilter on the countable Boolean group $\mathbb{B}$ is a PS-ultrafilter but not selective. If there exists a PS-ultrafilter on some countable Abelian group then there is a $P$-point in $\omega^*$. 

Clearly, an ultrafilter $\UU$ on $\mathbb{B}$ is a PS-ultrafilter if and only if $\UU$ is $\mathbb{B}$-Ramsey. Thus, a $\mathbb{B}$-Ramsey ultrafilter needs not to be selective, but such an ultrafilter cannot be constructed in ZFC with no additional assumptions.
\end{remark}

\section{Thin ultrafilters}
A free ultrafilter $\UU$ on $\omega$ is said to be
\begin{itemize}
  \item {\it $P$-point} if, for any partition $\mathcal{P}$ of $\omega$, either $A\in \UU$ for some cell $A$ of $\PP$ or there exists $U\in \UU$ such that $U\cap A$ is finite for each $A\in \PP$;
  \item {\it $Q$-point} if, for any partition $\PP$ of $\w$ into finite subsets, there exists $U\in \UU$ such that $|U\cap A| \leqslant 1$ for each $A\in \PP$.
\end{itemize}

Clearly, $\UU$ is selective if and only if $\UU$ is a $P$-point and a $Q$-point. It is well known that the existence of $P$- or $Q$-points is independent of the system of axioms ZFC.

We say that a free ultrafilter $\UU$ on $\w$ is a {\it $T$-point} if, for every countable group $G$ of permutations of $\w$, there is a thin subset $U\in\UU$ in the $G$-space $\w$.

To give a combinatorical characterization of $T$-points (see \cite{b8,b9}), we need some definitions.

A covering $\F$ of a set $X$ is called uniformly bounded if there exists $n\in \NN$ such that $|\cup\{F\in \F : x\in F\}|\leqslant n$ for each $x\in X$.

For a metric space $(X,d)$ and $n\in\NN$, we denote $B_d(x,n)=\{y\in X: d(x,y)\leqslant n\}$. A metric $d$ is called {\it locally finite (uniformly locally finite)} if, for every $n\in\NN$, $B_d(x,n)$ is finite for each $x\in X$ (there exists $c(n)\in\NN$ such that $|B_d(x,n)|\leqslant c(n)$ for each $x\in X$).

A subset $A$ of $(X,d)$ is called {\it $d$-thin} if, for every $n\in \NN$ there exists a bounded subset $B$ of $X$ such that $B_d(a,n)\cap A=\{a\}$ for each $a\in A\setminus B$.

\begin{theorem} For a free ultrafilter $\UU$ on $\w$, the following statement are equivalent:
\begin{itemize}
  \item[(i)] $\UU$ is a $T$-point;
  \item[(ii)] for any sequence $(\F_n)_{n\in\w}$ of uniformly bounded coverings of $\w$, there exists $U\in\UU$ such that, for each $n\in\w$, $|F\cap U|\leqslant 1$ for all but finitely many $F\in \F_n$;
  \item[(iii)] for each uniformly locally finite metric $d$ on $\w$, there is a $d$-thin member $U\in\UU$.
\end{itemize}
\end{theorem}

We do not know if a sequence of coverings in $(ii)$ can be replaced to a sequence of partitions.
\begin{remark}
By \cite[Theorem 3]{b10}, a free ultrafilter $\UU$ on $\omega$ in selective if and only if, for every metric $d$ on $\w$, there is a $d$-thin member of $\UU$.
\end{remark}

\begin{remark}
By \cite[Theorem 8]{b10}, a free ultrafilter $\UU$ on $\omega$ is a $Q$-point if and only if, for every locally finite metric $d$ on $\w$, there is a $d$-thin member of $\UU$.
\end{remark}

\begin{remark}
It is worth to be mentioned the following metric characterization of $P$-points: a free ultrafilter $\UU$ on $\w$ is a $P$-point if and only if, for every metric $d$ on $\w$, either some member of $\UU$ is bounded or there is $U\in \UU$ such that $(U,d)$ is locally finite.
\end{remark}

A free ultrafilter $\UU$ on $\w$ is said to be a {\it weak $P$-point} (a {\it NWD-point}) if $\UU$ is not a limit point of a countable subset in $\w^*$ (for every injective mapping $f:\w\to \mathbb{R}$, there is $U\in \UU$ such that $f(U)$ is nowhere dense in $\mathbb{R}$). We note that a weak $P$-point exists in ZFC.

In the next theorem, we summarize the main results from \cite{b8}.
\begin{theorem}
  Every $P$-point and every $Q$-point is a $T$-point. Under CH, there exists a $T$-point which is neither $P$-point, nor NWD-point, nor $Q$-point. For every ultrafilter $\mathcal{V}$ on $\w$, there exist a $T$-point $\UU$ and a mapping $f:\w\to\w$ such that $\mathcal{V}=f^\beta(\UU)$.
\end{theorem}
\begin{question}
Does there exist a $T$-point in ZFC?
\end{question}

\begin{question}
Is every weak $P$-point a $T$-point?
\end{question}

\begin{question}
(T. Banakh). Let $\UU$ be a free ultrafilter on $\w$ such that, for any metric $d$ on $\w$, some member of $\UU$ is discrete in $(X,d)$. Is $\UU$ a $T$-point?
\end{question}

A free ultrafilter $\UU$ on $\w$ is called a $T_{\aleph_0}$-point if, for each minimal well ordering $<$ of $\w$, there is a $d_<$-thin member of $\UU$, where $d_<$ is the natural metric on $\w$ defined by $<$. By Theorem 6.1, each $T$-point is $T_{\aleph_0}$-point.
\begin{question}
Is every $T_{\aleph_0}$-point a $T$-point? Does there exist a $T_{\aleph_0}$-point in ZFC?
\end{question}

\begin{remark}
  An ultrafilter $\UU$ on $\w$ is called {\it rapid} if, for any partition $\{P_n: n\in \w\}$ of $\w$ into finite subsets, there exists $U\in \UU$ such that $|U\cap P_n|\leqslant n$ for every $n\in \w$. Jana Fla\v{s}kov\'{a} (see \cite[p.350]{b10}) noticed that, in contrast to $Q$-points, a rapid ultrafilter needs not to be a $T$-point.
\end{remark}

\begin{remark}
  A family $\F$ of infinite subsets of $\w$ is {\it coideal} if $M\subseteq N, M\in\F \Rightarrow N\in \F$ and $M=N_0\cup N_1, M\in \F \Rightarrow N_0\in\F \vee N_1 \in \F$. Clearly, the family of all infinite subsets of $\w$ is a coideal.
\end{remark}

Following \cite{b27}, we say that a coideal F is
\begin{itemize}
  \item {\it $P$-coideal} if, for every decreasing sequence $(A_n)_{n\in\w}$ in $\F$ there is $B\in \F$ such that $B\setminus A_n$ is finite for each $n\in \w$;
  \item {\it $Q$-coideal} if, for every $A\in \F$ and every partition $A=\cup_{n\in\w}F_n$ with $F_n$ finite, there is $B\in\F$ such that $B\subseteq A$ and $|B\cap F_n|\leqslant 1$ for each $n\in\w$.
\end{itemize}

We say that a coideal $\F$ is a {\it $T$-coideal} if, for every countable group $G$ of permutations of $\w$ and every $M\in\F$ there exists a $G$-thin subset $N\in\F$ such that $N\subseteq M$.

Generalizing the first statement in Theorem 6.2, we get: every $P$-coideal and every $Q$-coideal is a $T$-coideal.

\begin{remark}
We say that $\UU \in \w^*$ is sparse (scattered) if, for every countable group $G$ of permutations of $\w$, there is sparse (scattered) in $(G,w)$ member of $\UU$. Clearly, $T$-point $\Rightarrow$ sparse ultrafilter $\Rightarrow$ scattered ultrafilter.
\end{remark}

\begin{question}
Does there exist sparse (scattcred) ultrafilter in ZFC? Is every weak $P$-point scattered ultrafilter?
\end{question}

\begin{question}
  Let $\UU$ be a free ultrafilter on $\w$ such that, for every countable group $G$ of permutations of $\w$, the orbit $\{g\UU: g\in G\}$ is discrete in $\w^*$. Is $\UU$ a weak $P$-point?
\end{question}

\section{The ballean context}

Following \cite{b21,b25}, we say that a {\it ball structure} is a triple $\mathcal{B}=(X,P,B)$, where $X$, $P$ are non-empty sets and, for every $x\in X$ and $\alpha\in P$, $B(x,\alpha)$ is a subset of $X$ which is called a \emph{ball of radius} $\alpha$ around $x$. It is supposed that $x\in B(x,\alpha)$ for all $x\in X$ and $\alpha\in P$.  The set $X$ is called the {\it support} of $\mathcal{B}$, $P$ is called the set of {\it radii}.

Given any $x\in X, A\subseteq X$ and  $\alpha\in P$ we set
$$B^*(x,\alpha)=\{y\in X:x\in B(y,\alpha)\},\
B(A,\alpha)=\bigcup_{a\in A}B(a,\alpha)$$

A ball structure $\mathcal{B}=(X,P,B)$ is called a {\it ballean} if
\begin{itemize}
\item for any $\alpha,\beta\in P$,
there exist $\alpha',\beta'$ such that, for every $x\in X$,
$$B(x,\alpha)\subseteq B^*(x,\alpha'),\ B^*(x,\beta)\subseteq B(x,\beta');$$
\item for any $\alpha,\beta\in P$,
there exists $\gamma\in P$ such that, for every $x\in X$,
$$B(B(x,\alpha),\beta)\subseteq B(x,\gamma);$$
\end{itemize}

A ballean $\mathcal{B}$ on $X$ can also be determined in terms of entourages of the diagonal of $X\times X$ ( in this case it is called a coarse structure \cite{b26}) and can be considered as an asymptotic counterpart of a uniform topological space.

Let $\mathcal{B}_1=(X_1,P_1,B_1)$, $\mathcal{B}_2=(X_2,P_2,B_2)$ be balleans. A mapping $f:X_1\rightarrow X_2$ is called a $\prec$-\emph{mapping} if, for every $\alpha\in P_1$, there exists $\beta\in P_2$ such that, for every $x\in X_1$,
$f(B_1(x,\alpha))\subseteq B_2(f(x),\beta)$. A bijection $f:X_1\rightarrow X_2$ is called an {\it asymorphism} if $f$ and $f^{-1}$ are $\prec$-mappings.

Every metric space $(X,d)$ defines the metric ballean $(X, \mathbb{R}^+, B_d)$, where $B_d(x,r)=\{y\in X: d(x,y) \leqslant r\}$. By \cite[Theorem 2.1.1]{b25}, a ballean $(X, P, B)$ is metrizable (i.e. asymorphic to some metric ballean) if and only if there exists a sequence $(\alpha_n)_{n\in\omega}$ in $P$ such that, for every $\alpha \in P$, one can find $n\in \omega$ such that $B(x,\alpha) \subseteq B(x, \alpha_n)$ for each $x\in X$.

Let $G$ be a group, $\mathcal{I}$ be an ideal in the Boolean algebra $\mathcal{P}_G$ of all subsets of $G$, i.e. $\varnothing \in \mathcal{I}$ and if $A,B\in\mathcal{I}$ and $A'\subseteq A$ then $A\cup B\in \mathcal{I}$ and $A'\in\mathcal{I}$. An ideal $\mathcal{I}$ is called a {\it group ideal} if, for all $A,B\in \mathcal{I}$, we have $AB\in \mathcal{I}$ and $A^{-1}\in\mathcal{I}$. For construction of group ideals see \cite{b16}.

For a $G$-space $X$ and a group ideal $\mathcal{I}$ on $G$, we define the ballean $\mathcal{B}(G,X,\mathcal{I})$ as the triple $(X,\mathcal{I},B)$ where $B(x,A)=Ax\cup \{x\}$. In the case where $\mathcal{I}$ is the ideal of all finite subsets of $G$, we omit $\mathcal{I}$ and return to the notation $B(x, A)$ used from the very beginning of the paper.

The following couple of theorems from \cite{b10, b15} demonstrate the tight interrelations between balleans and $G$-spaces.

\begin{theorem}
Every ballean $\B$ with the support $X$ is asymorphic to the ballean $\B(G,X,\I)$ for some subgroup $G$ of the group $S_X$ of all permutations of $X$ and some group ideal $\I$ on $G$.
\end{theorem}

\begin{theorem}
Every metrizable ballean $\B$ with the support $X$ is asymorphic to the ballean $\B(G,X,\I)$ for some subgroup $G$ of $S_X$ and some group ideal $\I$ on $G$ with countable base such that, for all $x,y\in X$, there is $A\in\I$ such that $y\in Ax$.
\end{theorem}

A ballean $\B=(X,P,B)$ is called {\it locally finite (uniformly locally finite) } if each ball $B(x,\alpha)$ is finite (for each $\alpha \in P$, there exists $n\in \mathbb{N}$ such that $|B(x, \alpha)| \leqslant n$ for every $x\in X$.
\begin{theorem}
Every locally finite ballean $\B$ with the support $X$ is asymorphic to the ballean $\B(G,X,\I)$ for some subgroup $G$ of $S_X$ and some group ideal $\I$ on $G$ with a base consisting of subsets compact in the topology of pointwise convergence on $S_X$.
\end{theorem}
\begin{theorem}\label{t_7_4}
Every uniformly locally finite ballean $\B$ with the support $X$ is asymorphic to the ballean $\B(G,X,[G]^{<\omega})$ for some subgroup $G$ of $S_X$.
\end{theorem}
We note that Theorem \ref{t_7_4} plays the key part in the proof of Theorem 6.1.

For ultrafilters on metric spaces and balleans we address the reader to \cite{b12, b20, b24}.

Department of Cybernetics

Taras Shevchenko National University,

Volodymyrs'ka St., 64,

01601 Kyiv, Ukraine

opetrenko72@gmail.com

i.v.protasov@gmail.com

\end{document}